\documentclass[11pt]{scrartcl}

\usepackage[utf8]{luainputenc}
\usepackage[USenglish]{babel}
\usepackage{csquotes}

\usepackage[a4paper,top=27mm,bottom=20mm,inner=20mm,outer=15mm]{geometry}

\usepackage[%
  backend=bibtex,bibencoding=ascii,
  style=numeric-comp,
  giveninits=true, uniquename=init, 
  natbib=true,
  url=true,
  doi=true,
  isbn=false,
  backref=false,
  maxnames=99,
  ]{biblatex}
\usepackage{amsmath}
\allowdisplaybreaks
\numberwithin{equation}{section}
\usepackage{amssymb}
\usepackage{commath}
\usepackage{mathtools}
\usepackage{bbm}
\usepackage{nicefrac}
\usepackage{subdepth}

\usepackage{siunitx}
\usepackage{algorithm}
\usepackage{algpseudocode}

\usepackage{amsthm}
\usepackage{thmtools}
\usepackage{etoolbox}
\makeatletter
\patchcmd{\thmt@setheadstyle}
 {\bgroup\thmt@space}
 {\thmt@space}
 {}{}
\patchcmd{\thmt@setheadstyle}
 {\egroup\fi}
 {\fi}
 {}{}
\makeatother
\declaretheoremstyle[
  bodyfont=\normalfont\itshape,
  headformat=\NAME\ \NUMBER\NOTE,
]{myplain}
\declaretheoremstyle[
  headformat=\NAME\ \NUMBER\NOTE,
]{mydefinition}
\newcommand{\envqed}{{\lower-0.3ex\hbox{$\triangleleft$}}}
\declaretheorem[style=myplain,numberwithin=section]{theorem}

\usepackage[plainpages=false,pdfpagelabels,hidelinks,unicode]{hyperref}

\usepackage{color}
\usepackage{graphicx}
\usepackage[small]{caption}
\usepackage{subcaption}

\begingroup\expandafter\expandafter\expandafter\endgroup
\expandafter\ifx\csname pdfsuppresswarningpagegroup\endcsname\relax
\else
\fi

\usepackage{booktabs}
\usepackage{rotating}
\usepackage{multirow}

\usepackage{enumitem}

\usepackage{ifluatex}
\ifluatex
  \usepackage[no-math]{fontspec}
\else
  \usepackage[T1]{fontenc}
\fi
\usepackage{newpxtext,newpxmath}

\usepackage{xparse}

\let\epsilon\varepsilon
\let\phi\varphi
\let\rho\varrho

\usepackage{xspace}

\renewcommand{\vec}[1]{\pmb{#1}}

\NewDocumentCommand{\opD}{m+g}{%
  \IfNoValueTF{#2}
    {D_{#1}}
    {D_{#1,#2}}%
}
\NewDocumentCommand{\opDsplit}{m+g}{%
  \IfNoValueTF{#2}
    {\widetilde{D}_{#1}}
    {\widetilde{D}_{#1,#2}}%
}
\NewDocumentCommand{\opM}{g}{%
  \IfNoValueTF{#1}
    {M}
    {M_{#1}}%
}
\NewDocumentCommand{\opQ}{g}{%
  \IfNoValueTF{#1}
    {Q}
    {Q_{#1}}%
}
\NewDocumentCommand{\opI}{g}{%
  \IfNoValueTF{#1}
    {I}
    {I_{#1}}%
}
\NewDocumentCommand{\opV}{g}{%
  \IfNoValueTF{#1}
    {V}
    {V_{#1}}%
}
\NewDocumentCommand{\opB}{g}{%
  \IfNoValueTF{#1}
    {B}
    {B_{#1}}%
}
\NewDocumentCommand{\opR}{g}{%
  \IfNoValueTF{#1}
    {R}
    {R_{#1}}%
}
\NewDocumentCommand{\opN}{m+g}{%
  \IfNoValueTF{#2}
    {N_{#1}}
    {N_{#1,#2}}%
}
\NewDocumentCommand{\x}{g}{%
  \IfNoValueTF{#1}
    {\vec{x}}
    {\vec{x}_{#1}}%
}

\NewDocumentCommand{\fnum}{g}{%
  \IfNoValueTF{#1}
    {f^{\mathrm{num}}}
    {f^{\mathrm{num,#1}}}%
}
\NewDocumentCommand{\vecfnum}{g}{%
  \IfNoValueTF{#1}
    {\vec{f}^{\mathrm{num}}}
    {\vec{f}^{\mathrm{num,#1}}}%
}
\NewDocumentCommand{\vecfcorr}{g}{%
  \IfNoValueTF{#1}
    {\vec{f}^{\mathrm{corr}}}
    {\vec{f}^{\mathrm{corr,#1}}}%
}
\NewDocumentCommand{\fvol}{g}{%
  \IfNoValueTF{#1}
    {f^{\smash{\mathrm{vol}}}}
    {f^{\smash{\mathrm{vol,#1}}}}%
}

\newcommand{\numflux}[1]{\widehat{#1}}

\newcommand{\mean}[1]{{\{\mkern-6mu\{}#1{\}\mkern-6mu\}}}

\newcommand{\jump}[1]{{[\mkern-3mu[}#1{]\mkern-3mu]}}

\newcommand{\orcid}[1]{ORCID:~\href{https://orcid.org/#1}{#1}}
\usepackage{authblk}

\newenvironment{keywords}{\par\textbf{Key words.}}{\par}
\newenvironment{AMS}{\par\textbf{AMS subject classification.}}{\par}

\title{A discontinuous Galerkin discretization of elliptic problems with improved
convergence properties using summation by parts operators} 
\date{February 24, 2023} 

\author[1]{Hendrik Ranocha\thanks{\orcid{0000-0002-3456-2277}}}
\affil[1]{Applied Mathematics, University of Hamburg, Germany}

\begin{document}

\maketitle

\begin{abstract}
\noindent
  Nishikawa (2007) proposed to reformulate the classical Poisson equation as a
steady state problem for a linear hyperbolic system. This results in optimal
error estimates for both the solution of the elliptic equation and its gradient.
However, it prevents the application of well-known solvers for elliptic
problems. We show connections to a discontinuous Galerkin (DG) method analyzed
by Cockburn, Guzmán, and Wang (2009) that is very difficult to implement in
general. Next, we demonstrate how this method can be implemented efficiently
using summation by parts (SBP) operators, in particular in the context of
SBP DG methods such as the DG spectral element method (DGSEM). The resulting scheme
combines nice properties of both the hyperbolic and the elliptic point of view,
in particular a high order of convergence of the gradients, which is one order
higher than what one would usually expect from DG methods for elliptic problems.

\end{abstract}



\section{Introduction}
\label{sec:introduction}

Solving a Poisson equation $- \Delta \phi = f$
in a bounded domain $\Omega \subset \mathbb{R}^d$ with appropriate boundary
conditions (BCs) is a key task in many scientific simulations.
Nishikawa \cite{nishikawa2007first} proposed to compute numerical solutions
as steady state limits of the hyperbolic system
\begin{equation}
\label{eq:FOHS}
\begin{aligned}
  \partial_t \phi - \nabla \cdot q &= f,
  \\
  \partial_t q - \frac{1}{T_r} \nabla \phi &= -\frac{1}{T_r} q,
\end{aligned}
\end{equation}
where $T_r > 0$ is a relaxation time that can be chosen to accelerate
the convergence to the steady state  \cite{nishikawa2018dimensional}.
Some earlier works on this ``hyperbolic heat equation'' are
\cite{cattaneo1958forme,nagy1994behavior,vanleer2001computational};
some later articles based on the idea are
\cite{nishikawa2010first,abgrall2014high,abgrall2015linear,desantis2015high}.

The ``hyperbolic diffusion'' approach enables optimal convergence not only of
the potential~$\phi$ but also of the gradient~$q$ for discontinuous Galerkin
(DG) methods. Moreover, it simplifies the coupling to hyperbolic equations in
multi-physics problems such as astrophysical fluid flows with self-gravity
\cite{schlottkelakemper2021purely}. However, it would be nice to keep the
superconvergence properties of the gradients in a classical elliptic formulation
to use state-of-the-art high-performance solvers
\cite{kronbichler2018performance,fehn2020hybrid}.
Focusing on DG methods, we will explain that the
steady-state formulation of \eqref{eq:FOHS} is equivalent to a scheme analyzed by
\citet{cockburn2009superconvergent}. However, this method appears to be
difficult to implement since
a linear system needs to be solved to compute the gradient
$q$, i.e., to evaluate the residual of the elliptic discretization. We will
explain how this difficulty can be solved for methods using summation by parts (SBP)
operators, in particular for discontinuous Galerkin spectral element methods
(DGSEM) using Gauss-Lobatto-Legendre nodes \cite{gassner2013skew}. See
\cite{manzanero2018bassi} for more observations how the SBP structure of DGSEM
can be used to analyze and improve DG methods for elliptic problems.

\section{Main result}
\label{sec:main}

We focus on 1D for simplicity. All results extend
to the multi-dimensional case using tensor product spaces, e.g., DGSEM.
The weak formulation of the steady state of \eqref{eq:FOHS} with test functions
$\psi, \chi$ on an interval $(x_i, x_{i+1})$ is
\begin{equation}
\label{eq:FOHS-upwind-as-elliptic-DG-weak}
  \int_{x_{i}}^{x_{i+1}} q \, \partial_x \chi \dif x
  =
  \int_{x_{i}}^{x_{i+1}} f \, \chi \dif x
  + \left[
    \numflux{q} \, \chi
  \right]_{x_{i}}^{x_{i+1}},
  \qquad
  \int_{x_{i}}^{x_{i+1}} q \, \psi \dif x
  =
  - \int_{x_{i}}^{x_{i+1}} \phi \, \partial_x \psi \dif x
  + \left[
    \numflux{\phi} \, \psi
  \right]_{x_{i}}^{x_{i+1}},
\end{equation}
where $\numflux{q}, \numflux{\phi}$ are numerical fluxes.
This steady state formulation is obtained by multiplying the second equation
by $T_r$. Thus, the system depends on the relaxation time $T_r$ only via the
numerical fluxes $\numflux{q}, \numflux{\phi}$ associated to the time-dependent
problem \eqref{eq:FOHS}.
The classical upwind numerical fluxes of the hyperbolic system
\eqref{eq:FOHS} are \cite{nishikawa2018dimensional}
\begin{equation}
\label{eq:FOHS-upwind-as-elliptic-fluxes}
  \numflux{\phi} = \mean{\phi} + \frac{\sqrt{T_r}}{2} \jump{q},
  \qquad
  \numflux{q} = \mean{q} + \frac{1}{2 \sqrt{T_r}} \jump{\phi},
\end{equation}
where $\mean{\cdot}$ denotes the arithmetic mean and $\jump{\cdot}$ the
jump at an interface.
In the context of the Poisson equation $-\Delta \phi = f$, these numerical
fluxes fit into the classical framework of \citet{arnold2002unified}; they
are consistent and conservative (and hence result in adjoint consistency).
\citet{castillo2000priori} analyzed similar numerical fluxes of the form
\begin{equation}
\label{eq:fluxes-cockburn2009superconvergent}
  \numflux{\phi} = \mean{\phi} - C_{12} \jump{\phi} + C_{22} \jump{q},
  \qquad
  \numflux{q} = \mean{q} + C_{11} \jump{\phi} + C_{12} \jump{q},
\end{equation}
which match \eqref{eq:FOHS-upwind-as-elliptic-fluxes} for $C_{12} = 0$. Using
polynomials of degree $p$, they proved that $(\phi, q)$ converge with orders
$(p + 1, p + 1/2)$ on general meshes if $|C_{12}|$, $|C_{11}|$, and $C_{22}$
are of order unity. \citet{cockburn2009superconvergent} extended this analysis
and proved that $(\phi, q)$ converge with optimal order $p + 1$ if $C_{11} > 0$,
$C_{22} > 0$, $C_{11} \propto 1 / C_{22}$, and $C_{11}, |C_{12}|$ are bounded,
which is exactly the situation for the numerical fluxes
\eqref{eq:FOHS-upwind-as-elliptic-fluxes}.
However, \citet{cockburn2009superconvergent} noted ``Of course, the DG methods
under consideration are difficult to implement'' since the numerical flux
$\numflux{\phi}$ used to compute the gradient $q$ depends on $q$. Thus,
a linear system needs to be solved to compute the residual of the elliptic
discretization.

In each element $e$, SBP operators \cite{svard2014review,fernandez2014review}
are given by
i) a discrete derivative operator $D_e$ approximating $\partial_x$,
ii) a diagonal mass/norm matrix $M_e$ approximating the $L^2$ inner product, and
iii) interpolation operators $t_{e,L/R}^T$ evaluating a numerical solution at
the left/right boundary of the element $e$.
Throughout, we use nodal approximations with grid nodes at the boundaries of
each element as in DGSEM. In this case, $t_{e,L}^T = (1, 0, \dots, 0)$ and
$t_{e,R}^T = (0, \dots, 0, 1)$. Furthermore, we require the SBP condition
$M_e D_e + (M_e D_e)^T = t_{e,R} t_{e,R}^T - t_{e,L} t_{e,L}^T$ mimicking
integration by parts \cite{gassner2013skew}.

In this framework, the DG discretization \eqref{eq:FOHS-upwind-as-elliptic-DG-weak}
can be written in the equivalent strong form as
\begin{equation}
\begin{aligned}
  q_{e}
  &=
  D_{e} \phi_{e}
  + M_{e}^{-1} t_{e,R} \left(
    \numflux{\phi} - t_{e,R}^T \phi_{e}
  \right)
  - M_{e}^{-1} t_{e,L} \left(
    \numflux{\phi} - t_{e,L}^T \phi_{e}
  \right),
  \\
  -D_{e} q_{e}
  &=
  f
  + M_{e}^{-1} t_{e,R} \left(
    \numflux{q} - t_{e,R}^T q_{e}
  \right)
  - M_{e}^{-1} t_{e,L} \left(
    \numflux{q} - t_{e,L}^T q_{e}
  \right).
\end{aligned}
\end{equation}
Here, $q_e$, $\phi_e$, and $f$ are the vectors of coefficients representing the
respective polynomials in element $e$ in the chosen nodal basis.
Now, we are prepared to formulate the main result of this short note.

\begin{theorem}
\label{thm:q-local}
  The discrete gradient $q$ of \eqref{eq:FOHS-upwind-as-elliptic-DG-weak} with
  numerical fluxes \eqref{eq:FOHS-upwind-as-elliptic-fluxes} can be evaluated
  locally using only surface values of $\phi$ and $D \phi$ from neighboring
  elements if diagonal-norm SBP operators including the boundaries are used.
  In particular, DGSEM is included in this class of SBP operators.
\end{theorem}
\begin{proof}
  It suffices to consider two elements indicated by subscripts $l, r$.
  Abbreviating surface terms not belonging to their common interface as
  $ST_{l,L}$ (left element, left surface) and
  $ST_{r,R}$ (right element, right surface),
  the corresponding discretizations are
  \begin{equation}
  \label{eq:ql-and-qr-implicit}
  \begin{aligned}
    q_{l}
    &=
    D_{l} \phi_{l}
    + M_{l}^{-1} t_{l,R} \left(
      \frac{1}{2} (t_{r,L}^T \phi_{r} - t_{l,R}^T \phi_{l})
      + \frac{\sqrt{T_r}}{2} (t_{r,L}^T q_{r} - t_{l,R}^T q_{l})
    \right)
    + ST_{l,L},
    \\
    q_{r}
    &=
    D_{r} \phi_{r}
    - M_{r}^{-1} t_{r,L} \left(
      -\frac{1}{2} (t_{r,L}^T \phi_{r} - t_{l,R}^T \phi_{l})
      + \frac{\sqrt{T_r}}{2} (t_{r,L}^T q_{r} - t_{l,R}^T q_{l})
    \right)
    + ST_{r,R}.
  \end{aligned}
  \end{equation}
  Again, $q_{l/r}$ and $\phi_{l/r}$ are the vectors of coefficients
  representing the respective polynomials in the left/right element in
  the chosen nodal basis.
  Since boundary nodes are included and the mass matrix is diagonal, the surface
  terms vanish everywhere except at their corresponding interface nodes.
  In particular, the restriction of $q_{l}$ to the right surface of element $l$
  is not influenced by the left surface term $ST_{l,L}$.
  Hence, the jump of boundary values of $q$ is
  \begin{equation}
  \begin{aligned}
    t_{r,L}^T q_{r} - t_{l,R}^T q_{l}
    &=
    t_{r,L}^TD_{r} \phi_{r}
    - t_{r,L}^T M_{r}^{-1} t_{r,L} \left(
      -\frac{1}{2} (t_{r,L}^T \phi_{r} - t_{l,R}^T \phi_{l})
      + \frac{\sqrt{T_r}}{2} (t_{r,L}^T q_{r} - t_{l,R}^T q_{l})
    \right)
    \\
    &\quad
    - t_{l,R}^T D_{l} \phi_{l}
    - t_{l,R}^T M_{l}^{-1} t_{l,R} \left(
      \frac{1}{2} (t_{r,L}^T \phi_{r} - t_{l,R}^T \phi_{l})
      + \frac{\sqrt{T_r}}{2} (t_{r,L}^T q_{r} - t_{l,R}^T q_{l})
    \right).
  \end{aligned}
  \end{equation}
  This equation can be solved for the jump of interface values of $q$,
  \begin{equation}
    t_{r,L}^T q_{r} - t_{l,R}^T q_{l}
    =
    - c_1 (t_{r,L}^T \phi_{r} - t_{l,R}^T \phi_{l})
    + c_2 (t_{r,L}^T D_{r} \phi_{r} - t_{l,R}^T D_{l} \phi_{l}),
  \end{equation}
  where
  \begin{equation}
    c_1
    =
    \frac{1}{2} \frac{t_{r,L}^T M_{r}^{-1} t_{r,L} - t_{l,R}^T M_{l}^{-1} t_{l,R}}
                     {1 + \frac{\sqrt{T_r}}{2} (t_{r,L}^T M_{r}^{-1} t_{r,L} + t_{l,R}^T M_{l}^{-1} t_{l,R})},
    \quad
    c_2
    =
    \frac{1}
         {1 + \frac{\sqrt{T_r}}{2} (t_{r,L}^T M_{r}^{-1} t_{r,L} + t_{l,R}^T M_{l}^{-1} t_{l,R})}.
  \end{equation}
  Note that $c_1$ vanishes for uniform grids with symmetric quadrature rules.
  In general, $c_1$ and $c_2$ depend on the grid spacing.
  Inserting this expression of the jump of $q$ at the interface into
  \eqref{eq:ql-and-qr-implicit} yields
  \begin{equation}
  \begin{aligned}
    q_{l}
    &=
    D_{l} \phi_{l}
    + M_{l}^{-1} t_{l,R} \left(
      \frac{1}{2} (1 - \sqrt{T_r} c_{1,R}) (t_{r,L}^T \phi_{r} - t_{l,R}^T \phi_{l})
      + \frac{\sqrt{T_r}}{2} c_{2,R} (t_{r,L}^T D_{r} \phi_{r} - t_{l,R}^T D_{l} \phi_{l})
    \right)
    + ST_{l,L},
    \\
    q_{r}
    &=
    D_{r} \phi_{r}
    + M_{r}^{-1} t_{r,L} \left(
      \frac{1}{2} (1 + \sqrt{T_r} c_{1,L}) (t_{r,L}^T \phi_{r} - t_{l,R}^T \phi_{l})
      - \frac{\sqrt{T_r}}{2} c_{2,L} (t_{r,L}^T D_{r} \phi_{r} - t_{l,R}^T D_{l} \phi_{l})
    \right)
    + ST_{r,R}.
  \end{aligned}
  \end{equation}
  To sum up, the gradient in an interior element $e$ can be computed explicitly as
  \begin{equation}
  \label{eq:q-interior}
  \begin{aligned}
    q_{e}
    =
    D_{e} \phi_{e}
    &+ M_{e}^{-1} t_{e,R} \left(
      \frac{1}{2} (1 - \sqrt{T_r} c_{1,R}) \jump{\phi}_R
      + \frac{\sqrt{T_r}}{2} c_{2,R} \jump{D \phi}_R
    \right)
    \\
    &+ M_{e}^{-1} t_{e,L} \left(
      \frac{1}{2} (1 + \sqrt{T_r} c_{1,L}) \jump{\phi}_L
      - \frac{\sqrt{T_r}}{2} c_{2,L} \jump{D \phi}_L
    \right).
  \end{aligned}
  \end{equation}
  At a boundary point where a Dirichlet condition is imposed weakly for the
  potential $\phi$, the numerical fluxes based on an energy analysis for the
  hyperbolic diffusion system are the ones used by
  \citet{cockburn2009superconvergent}, i.e.,
  \begin{equation}
    \numflux{\phi} = \phi^\mathrm{boundary},
    \qquad
    \numflux{q} = q^\mathrm{interior} + \frac{1}{2 \sqrt{T_r}} \jump{\phi}.
  \end{equation}
  Hence, the numerical flux $\numflux{\phi}$ does not depend on $q$ at a Dirichlet
  boundary and no special care is needed. Thus, the discretization in the element
  at the left boundary is
  \begin{equation}
  \label{eq:q-left-boundary}
    q_{e}
    =
    D_{e} \phi_{e}
    + M_{e}^{-1} t_{e,R} \left(
      \frac{1}{2} (1 - \sqrt{T_r} c_{1,R}) \jump{\phi}_R
      + \frac{\sqrt{T_r}}{2} c_{2,R} \jump{D \phi}_R
    \right)
    + M_{e}^{-1} t_{e,L} (t_{e,L}^T \phi_{e} - \phi^\mathrm{BC}).
  \end{equation}
  Similarly, the gradient in the element at the right boundary is
  \begin{equation}
  \label{eq:q-right-boundary}
    q_{e}
    =
    D_{e} \phi_{e}
    + M_{e}^{-1} t_{e,R} (\phi^\mathrm{BC} - t_{e,R}^T \phi_{e})
    + M_{e}^{-1} t_{e,L} \left(
      \frac{1}{2} (1 + \sqrt{T_r} c_{1,L}) \jump{\phi}_L
      - \frac{\sqrt{T_r}}{2} c_{2,L} \jump{D \phi}_L
    \right).
  \end{equation}
  Thus, the gradient can be computed locally using surface values of
  $\phi$ and $D \phi$ from neighboring elements.
\end{proof}

\section{Numerical experiments}
\label{sec:numerical}

We demonstrate the convergence properties of the method for
several Poisson problems summarized in Table~\ref{tab:setups}.
The right-hand side $f$ is chosen based on the solution $\varphi$.
We use Dirichlet BCs for non-periodic setups and vanishing mean values
of $\phi$ for periodic BCs.
The non-periodic 2D setup is taken from \cite{cockburn2009superconvergent}.
We choose the relaxation time $T_r$ as recommended in
\cite{nishikawa2018dimensional}, i.e., $T_r = L_r^2$ where the reference
length scale $L_r$ is set to
$L_r = (x_{\mathrm{max}} - x_{\mathrm{min}}) / (2 \pi)$
for a 1D interval $(x_{\mathrm{min}}, x_{\mathrm{max}})$ and
\begin{equation*}
  L_r = \frac{1}{2 \pi} \frac{(x_{\mathrm{max}} - x_{\mathrm{min}}) (y_{\mathrm{max}} - y_{\mathrm{min}})}{\sqrt{(x_{\mathrm{max}} - x_{\mathrm{min}})^2 + (y_{\mathrm{max}} - y_{\mathrm{min}})^2}}
\end{equation*}
for a 2D rectangle $(x_{\mathrm{min}}, x_{\mathrm{max}}) \times (y_{\mathrm{min}}, y_{\mathrm{max}})$.

\begin{table}[!htb]
\centering
\caption{Summary of the numerical experiment setups.}
\label{tab:setups}
  \begin{tabular}{c c c l c}
    \toprule
    \multicolumn{1}{c}{Setup} &
      \multicolumn{1}{c}{Dim.} &
      \multicolumn{1}{c}{Domain} &
      \multicolumn{1}{c}{Solution} &
      \multicolumn{1}{c}{Boundary Condition}
    \\
    \midrule
    Setup 1 & 1D & $(-1, 1)$ & $\phi(x) = \exp(-10 x^2)$ & Dirichlet
    \\
    Setup 2 & 1D & $(-2, 2)$ & $\phi(x) = \exp(-10 x^2) - \sqrt{\pi / 10} \operatorname{erf}(2 \sqrt{10}) / 4$ & periodic
    \\
    Setup 3 & 2D & $(-0.5, 0.5)^2$ & $\phi(x, y) = \cos(\pi x) \cos(\pi y)$ & Dirichlet
    \\
    Setup 4 & 2D & $(-1, 1)^2$ & $\phi(x, y) = 2 \cos(\pi x) \sin(2 \pi y)$ & periodic
    \\
    \bottomrule
  \end{tabular}
\end{table}

All methods are implemented in Julia \cite{bezanson2017julia}. We use
Trixi.jl \cite{ranocha2022adaptive,schlottkelakemper2021purely}
to compute steady state solutions of the hyperbolic system \eqref{eq:FOHS} and
SummationByPartsOperators.jl \cite{ranocha2021sbp} to implement the
corresponding elliptic approach. All source code required to
reproduce the numerical experiments is available online
\cite{ranocha2023discontinuousRepro}.

\begin{table}[!hbt]
\centering\small
  \caption{Numerical results of the 1D convergence experiments summarized in
           Table~\ref{tab:setups}.}
  \label{tab:1D}
  \begin{subfigure}{0.49\linewidth}
  \centering
  \setlength{\tabcolsep}{0.9ex}
    \caption{Convergence results for setup 1 with $p = 2$.}
    \begin{tabular}{rrrrr}
      \toprule
      \multicolumn{1}{c}{$N$} &
        \multicolumn{1}{c}{Error $\phi$} &
        \multicolumn{1}{c}{EOC} &
        \multicolumn{1}{c}{Error $q$} &
        \multicolumn{1}{c}{EOC} \\
      \midrule
       10 & 2.42e-02 &      & 6.88e-02 &      \\
       20 & 3.16e-03 & 2.94 & 8.64e-03 & 2.99 \\
       40 & 3.97e-04 & 2.99 & 1.08e-03 & 3.01 \\
       80 & 4.96e-05 & 3.00 & 1.34e-04 & 3.00 \\
      160 & 6.19e-06 & 3.00 & 1.67e-05 & 3.00 \\
      \bottomrule
    \end{tabular}
  \end{subfigure}%
  \vspace{\fill}%
  \begin{subfigure}{0.49\linewidth}
  \centering
  \setlength{\tabcolsep}{0.9ex}
    \caption{Convergence results for setup 1 with $p = 3$.}
    \begin{tabular}{rrrrr}
      \toprule
      \multicolumn{1}{c}{$N$} &
        \multicolumn{1}{c}{Error $\phi$} &
        \multicolumn{1}{c}{EOC} &
        \multicolumn{1}{c}{Error $q$} &
        \multicolumn{1}{c}{EOC} \\
      \midrule
       10 & 2.54e-03 &      & 6.73e-03 &      \\
       20 & 1.60e-04 & 3.99 & 4.17e-04 & 4.01 \\
       40 & 1.01e-05 & 3.99 & 2.60e-05 & 4.00 \\
       80 & 6.31e-07 & 4.00 & 1.62e-06 & 4.00 \\
      160 & 3.94e-08 & 4.00 & 1.01e-07 & 4.00 \\
      \bottomrule
    \end{tabular}
  \end{subfigure}%
  \\[1em]
  \begin{subfigure}{0.49\linewidth}
  \centering
  \setlength{\tabcolsep}{0.9ex}
    \caption{Convergence results for setup 2 with $p = 2$.}
    \begin{tabular}{rrrrr}
      \toprule
      \multicolumn{1}{c}{$N$} &
        \multicolumn{1}{c}{Error $\phi$} &
        \multicolumn{1}{c}{EOC} &
        \multicolumn{1}{c}{Error $q$} &
        \multicolumn{1}{c}{EOC} \\
      \midrule
       10 & 1.24e-01 &      & 4.36e-01 &      \\
       20 & 4.60e-02 & 1.43 & 6.45e-02 & 2.76 \\
       40 & 6.08e-03 & 2.92 & 8.07e-03 & 3.00 \\
       80 & 7.66e-04 & 2.99 & 1.00e-03 & 3.01 \\
      160 & 9.58e-05 & 3.00 & 1.25e-04 & 3.01 \\
      \bottomrule
    \end{tabular}
  \end{subfigure}%
  \vspace{\fill}%
  \begin{subfigure}{0.49\linewidth}
  \centering
  \setlength{\tabcolsep}{0.9ex}
    \caption{Convergence results for setup 2 with $p = 3$.}
    \begin{tabular}{rrrrr}
      \toprule
      \multicolumn{1}{c}{$N$} &
        \multicolumn{1}{c}{Error $\phi$} &
        \multicolumn{1}{c}{EOC} &
        \multicolumn{1}{c}{Error $q$} &
        \multicolumn{1}{c}{EOC} \\
      \midrule
       10 & 8.09e-02 &      & 1.06e-01 &      \\
       20 & 4.91e-03 & 4.04 & 6.33e-03 & 4.07 \\
       40 & 3.12e-04 & 3.98 & 3.91e-04 & 4.02 \\
       80 & 1.97e-05 & 3.99 & 2.44e-05 & 4.00 \\
      160 & 1.23e-06 & 4.00 & 1.52e-06 & 4.00 \\
      \bottomrule
    \end{tabular}
  \end{subfigure}%
\end{table}

We compute the $L^2$ errors of the numerical solutions on meshes with $N$
elements per coordinate direction using the Gauss-Lobatto-Legendre quadrature
rule associated with the DGSEM operators. The 1D results including the
experimental order of convergence (EOC) are shown in
Table~\ref{tab:1D}. These results are computed using direct sparse solvers
distributed with Julia \cite{bezanson2017julia}. Clearly, both the potential
$\phi$ and the gradient $q$ converge with optimal order $p + 1$ for polynomials
of degree~$p$. The results obtained by evolving the hyperbolic system
\eqref{eq:FOHS} match the results obtained from the elliptic implementation
and are thus not shown.

\begin{table}[!hbt]
\centering\small
  \caption{Numerical results of the 2D convergence experiments summarized in
           Table~\ref{tab:setups}.}
  \label{tab:2D}
  \begin{subfigure}{0.49\linewidth}
  \centering
  \setlength{\tabcolsep}{0.9ex}
    \caption{Convergence results for setup 3 with $p = 2$.}
    \begin{tabular}{rrrrrrr}
      \toprule
      \multicolumn{1}{c}{$N$} &
        \multicolumn{1}{c}{Error $\phi$} &
        \multicolumn{1}{c}{EOC} &
        \multicolumn{1}{c}{Error $q_1$} &
        \multicolumn{1}{c}{EOC} &
        \multicolumn{1}{c}{Error $q_2$} &
        \multicolumn{1}{c}{EOC} \\
      \midrule
         4 & 7.49e-02 &      & 3.34e-01 &      & 3.34e-01 &      \\
         8 & 2.51e-03 & 4.90 & 1.15e-02 & 4.86 & 1.15e-02 & 4.86 \\
        16 & 1.70e-04 & 3.88 & 8.48e-04 & 3.77 & 8.48e-04 & 3.77 \\
        32 & 1.46e-05 & 3.54 & 8.26e-05 & 3.36 & 8.26e-05 & 3.36 \\
        64 & 1.46e-06 & 3.33 & 9.23e-06 & 3.16 & 9.23e-06 & 3.16 \\
      \bottomrule
    \end{tabular}
  \end{subfigure}%
  \vspace{\fill}%
  \begin{subfigure}{0.49\linewidth}
  \centering
  \setlength{\tabcolsep}{0.9ex}
    \caption{Convergence results for setup 3 with $p = 3$.}
    \begin{tabular}{rrrrrrr}
      \toprule
      \multicolumn{1}{c}{$N$} &
        \multicolumn{1}{c}{Error $\phi$} &
        \multicolumn{1}{c}{EOC} &
        \multicolumn{1}{c}{Error $q_1$} &
        \multicolumn{1}{c}{EOC} &
        \multicolumn{1}{c}{Error $q_2$} &
        \multicolumn{1}{c}{EOC} \\
      \midrule
         4 & 1.81e-04 &      & 1.29e-03 &      & 1.29e-03 &      \\
         8 & 1.48e-05 & 3.61 & 9.83e-05 & 3.71 & 9.83e-05 & 3.71 \\
        16 & 8.60e-07 & 4.10 & 6.05e-06 & 4.02 & 6.05e-06 & 4.02 \\
        32 & 4.99e-08 & 4.11 & 3.69e-07 & 4.03 & 3.69e-07 & 4.03 \\
        64 & 3.04e-09 & 4.04 & 2.56e-08 & 3.85 & 2.56e-08 & 3.85 \\
      \bottomrule
    \end{tabular}
  \end{subfigure}%
  \\[1em]
  \begin{subfigure}{0.49\linewidth}
  \centering
  \setlength{\tabcolsep}{0.9ex}
    \caption{Convergence results for setup 4 with $p = 2$.}
    \begin{tabular}{rrrrrrr}
      \toprule
      \multicolumn{1}{c}{$N$} &
        \multicolumn{1}{c}{Error $\phi$} &
        \multicolumn{1}{c}{EOC} &
        \multicolumn{1}{c}{Error $q_1$} &
        \multicolumn{1}{c}{EOC} &
        \multicolumn{1}{c}{Error $q_2$} &
        \multicolumn{1}{c}{EOC} \\
      \midrule
         4 & 1.31e+00 &      & 4.05e+00 &      & 4.36e+00 &      \\
         8 & 1.55e-01 & 3.08 & 4.90e-01 & 3.05 & 7.01e-01 & 2.63 \\
        16 & 2.17e-02 & 2.84 & 6.86e-02 & 2.84 & 9.31e-02 & 2.91 \\
        32 & 2.83e-03 & 2.94 & 8.94e-03 & 2.94 & 1.19e-02 & 2.96 \\
        64 & 3.60e-04 & 2.97 & 1.14e-03 & 2.98 & 1.51e-03 & 2.98 \\
      \bottomrule
    \end{tabular}
  \end{subfigure}%
  \vspace{\fill}%
  \begin{subfigure}{0.49\linewidth}
  \centering
  \setlength{\tabcolsep}{0.9ex}
    \caption{Convergence results for setup 4 with $p = 3$.}
    \begin{tabular}{rrrrrrr}
      \toprule
      \multicolumn{1}{c}{$N$} &
        \multicolumn{1}{c}{Error $\phi$} &
        \multicolumn{1}{c}{EOC} &
        \multicolumn{1}{c}{Error $q_1$} &
        \multicolumn{1}{c}{EOC} &
        \multicolumn{1}{c}{Error $q_2$} &
        \multicolumn{1}{c}{EOC} \\
      \midrule
         4 & 8.67e-02 &      & 2.75e-01 &      & 8.63e-01 &      \\
         8 & 1.39e-02 & 2.64 & 4.37e-02 & 2.65 & 5.96e-02 & 3.86 \\
        16 & 9.35e-04 & 3.89 & 2.94e-03 & 3.89 & 3.90e-03 & 3.93 \\
        32 & 6.01e-05 & 3.96 & 1.89e-04 & 3.96 & 2.48e-04 & 3.97 \\
        64 & 3.81e-06 & 3.98 & 1.20e-05 & 3.98 & 1.56e-05 & 3.99 \\
      \bottomrule
    \end{tabular}
  \end{subfigure}%
\end{table}

The 2D results are shown in Table~\ref{tab:2D}. These results are computed using
the conjugate gradients (CG) implementation of Krylov.jl \cite{montoison2020krylov}
using matrix-free operators based on the interface of LinearOperators.jl
\cite{orban2020linearoperators}. Again, both the potential $\phi$ and the
gradient $q$ converge with optimal order.

\section*{Acknowledgments}

Special thanks to Jesse Chan for discussions related to this manuscript and
comments on an early draft.

\printbibliography

\end{document}